\documentclass[12pt,a4paper]{amsart}

\usepackage{amsmath,amsfonts,amssymb,mathtools,extarrows}
\usepackage{xcolor}

\usepackage{tikz}
\usepackage{color}
\usepackage{graphicx}

\usepackage{enumerate}  

\usepackage{cancel}

\usepackage[hmargin=2cm,vmargin=2cm]{geometry}

\usepackage[hypertexnames=false,hyperfootnotes=false,colorlinks=true,linkcolor=blue,%
citecolor=purple,filecolor=magenta,urlcolor=cyan,unicode,linktocpage=true,pagebackref=false]{hyperref}
\usepackage{nameref,zref-xr}     

\newcommand{\mbZ}{\mathbb Z}
\newcommand{\mbC}{\mathbb C}
\newcommand{\mbR}{\mathbb R}

\newcommand{\ord}{\operatorname{ord}}
\newcommand{\oM}{\overline{\mathcal M}}

\def\cM{{\mathcal{M}}}
\def\oM{{\overline{\mathcal{M}}}}
\def\CP{{{\mathbb C}{\mathbb P}}}

\def\d{{\partial}}

\newcommand{\Coef}{\mathrm{Coef}}

\newcommand{\cR}{\mathcal{R}}

\DeclareMathOperator{\res}{{res}}

\newcommand{\oB}{\overline{B}}
\newcommand{\codim}{\operatorname{codim}}
\newcommand{\spn}{\operatorname{span}}

\newcommand{\oH}{{\overline{\mathcal{H}}}}
\newcommand{\cH}{\mathcal{H}}

\newcommand{\mbCP}{\mathbb{C}\mathbb{P}}

\newcommand{\PGL}{\mathrm{PGL}}

\newcommand{\Hur}{\mathrm{Hur}}

\newcommand{\ov}{\overline{v}}
\newcommand{\ttheta}{\widetilde{\theta}}
\newcommand{\tA}{\widetilde{A}}
\newcommand{\SL}{\mathrm{SL}}

\newtheorem{theorem}{Theorem}[section]
\newtheorem{proposition}[theorem]{Proposition}
\newtheorem{lemma}[theorem]{Lemma}

\theoremstyle{remark}

\newtheorem{remark}[theorem]{Remark}
\newtheorem{example}[theorem]{Example}

\theoremstyle{definition}

\usepackage{color}

\renewcommand{\gg}[2]{\fill[color=white] (#2) circle(2mm) node {\color{black}$\substack{#1}$}; \draw (#2) circle (2mm)}
\newcommand{\lab}[4]{\draw (#1)++(#2:#3) node {$\substack{#4}$};}

\newcommand{\legm}[3]{\begin{scope}[shift={(#1)}] \draw (0:0) -- (#2:7.8mm);\fill[color=white] (#2:7.8mm) circle(1.9mm) node {\color{black}$\substack{#3}$};\end{scope}}

\newcommand{\legmrl}[3]{\begin{scope}[shift={(#1)}] \draw[decorate, decoration={zigzag, segment length=1.3mm,amplitude=0.5mm}] (0:0) -- (#2:7.8mm);\fill[color=white] (#2:7.8mm) circle(1.9mm) node {\color{black}$\substack{#3}$};\end{scope}}

\numberwithin{equation}{section}

\begin{document}

\usetikzlibrary { decorations.pathmorphing, decorations.pathreplacing, decorations.shapes, }
\tikz{\coordinate (A) at (0,0);\coordinate (B) at (0.8,0);\coordinate (C) at (1.6,0);\coordinate (D) at (2.4,0);\coordinate (A0) at (0,0.8);\coordinate (B0) at (1.6,0.8);\coordinate (B1) at (0,-0.8);\coordinate (B1a) at (-0.8,-0.8);\coordinate (C1b) at (0.8,-0.8);\coordinate (AB1) at (0.4,0.18);\coordinate (AB2) at (0.4,-0.18);\coordinate (E) at (-0.69,0.4);\coordinate (F) at (-0.69,-0.4);\coordinate (BC1) at (1.2,0.18);\coordinate (BC2) at (1.2,-0.18);\coordinate (G) at (-1.38,0);\coordinate (H) at (-2.18,0);\coordinate (I) at (-2.98,0);}

\title[Counting meromorphic differentials on $\CP^1$]{Counting meromorphic differentials on $\CP^1$}

\author{Alexandr Buryak}
\address{A. Buryak:\newline 
Faculty of Mathematics, National Research University Higher School of Economics, \newline
6 Usacheva str., Moscow, 119048, Russian Federation;\smallskip\newline 
Center for Advanced Studies, Skolkovo Institute of Science and Technology, \newline
1 Nobel str., Moscow, 143026, Russian Federation}
\email{aburyak@hse.ru}

\author{Paolo Rossi}
\address{P. Rossi:\newline Dipartimento di Matematica ``Tullio Levi-Civita'', Universit\`a degli Studi di Padova,\newline
Via Trieste 63, 35121 Padova, Italy}
\email{paolo.rossi@math.unipd.it}

\begin{abstract}
We give explicit formulas for the number of meromorphic differentials on $\mbCP^1$ with two zeros and any number of residueless poles and for the number of meromorphic differentials on~$\mbCP^1$ with one zero, two poles with unconstrained residue and any number of residueless poles, in terms of the orders of their zeros and poles. These are the only two finite families of differentials on $\mbCP^1$ with vanishing residue conditions at a subset of poles, up to the action of~$\PGL(2,\mbC)$. The first family of numbers is related to triple Hurwitz numbers by simple integration and we show its connection with the representation theory of~$\SL_2(\mbC)$ and the equations of the dispersionless KP hierarchy. The second family has a very simple generating series, and we recover it through surprisingly involved computations using intersection theory of moduli spaces of curves and differentials.
\end{abstract}

\date{\today}

\maketitle

\section{Introduction}

On the complex projective line, for any configuration of $n\geq 3$ distinct marked points and $n$ nonzero integers summing to $-2$, there exists a differential, unique up to multiplication by a nonzero complex constant, whose zeros and poles are at the $n$ marked points and their order is given by the $n$ integers.

\medskip

Up to the action of $\PGL(2,\mbC)$, which is $3$-dimensional, the space of configurations of $n$ marked points on $\mbCP^1$ has dimension $n-3$, so the number of configurations of $n$ marked points supporting a meromorphic differential with fixed orders of zeros and poles is finite if and only if $n=3$. Considering that, for $n=3$, one pole and one zero must always exist, this leaves us with two cases in which the number of meromorphic differentials is finite: two zeros and one pole or two poles and one zero.

\medskip

Each of these two cases can be enriched if we allow for extra poles of degree at least $-2$ whose residue is constrained to vanish. This way, for each new pole, we introduce a new degree of freedom (its position on $\mbCP^1$) and an extra equation (the vanishing of its residue), keeping each family zero-dimensional. We refer to differentials with two zeros and any number of residueless poles as {\it differentials of the first type}, and to differentials with one zero, two poles with unconstrained residue and any number of residueless poles as {\it differentials of the second type}. This terminology is not to be confused with the classical one for Abelian differential of the first, second, or third kind (they are actually quite incompatible).

\medskip

Meromorphic differentials of the first type are actually exact and can be integrated to meromorphic functions, so the first family of numbers is a special family of triple Hurwitz numbers for which classical Hurwitz techniques are available, as done in~\cite{CMSZ20}. Our results, beside recovering those of~\cite{CMSZ20}, relate these numbers with the representation theory of $\SL_2(\mbC)$ and the theory of integrable systems of PDEs. In particular, the numbers of differentials of the first type turn out to coincide, up to a simple factor, with the coefficients of the dispersionless Kadomtsev--Petviashvili (KP) equations. This relies on a more general result from \cite{BRZ21} involving Hodge and double ramification integrals on the space of residueless meromorphic differentials on genus $g$ curves and their relation with the coefficients of the equations of the full KP hierarchy. We also show that the generating series of the numbers of meromorphic differentials of the first type is a Dubrovin--Frobenius potential and present two explicit formulas for it. The dispersionless KP hierarchy is a reduction of the system of primary flows of the associated principal hierarchy. We thus view the constructed Dubrovin--Frobenius manifold as a natural Dubrovin--Frobenius manifold underlying the dispersionless KP hierarchy.   

\medskip

The number of differentials of the second type as a function of the order of their zeros and poles was computed in \cite[Proposition 2.3]{CC19} and in \cite[Proposition 4.6]{GT22}. It turns out that these numbers, or their generating function, are described by a remarkably simple formula, for which we give a new proof using techniques of intersection theory on the moduli space of stable curves and projectivized meromorphic differentials from \cite{BCGGM18,BCGGM19,CMZ20}. In particular, our proof is based on Lemma \ref{lemma:WDVV}, a recursive relation between numbers of meromorphic differentials of the first and second type of independent interest, lifted from the WDVV relations in the cohomology of the moduli spaces of rational stable curves.

\medskip

\subsection*{Acknowledgements}

We would like to thank Renzo Cavalieri and Takashi Takebe for useful discussions and Q. Gendron and J. Schmitt for providing us with useful references.

\medskip

The work of A.~B. is an output of a research project implemented as part of the Basic Research Program at the National Research University Higher School of Economics (HSE University). P.~R.~was partially supported by the BIRD-SID-2021 UniPD grant and is affiliated to GNSAGA-INDAM.

\medskip

\section{Main definitions}

With the same notations used in~\cite[Section~1]{BRZ21}, for $A=(a_1,\ldots,a_k) \in \mbZ^k$ and $B=(b_1,\ldots,b_n) \in \mbZ_{\le -2}^n$ satisfying $\sum a_i+\sum b_j=2g-2$, let $\cH_g(A;B)$ be the locus in~$\cM_{g,k+n}$ whose points correspond to genus $g$ smooth curves with marked points $z_1,\ldots,z_{k+n}$ such that there exists a meromorphic differential whose divisor of zeros and poles is $\sum_{i=1}^k a_i[z_i] + \sum_{j=1}^{n} b_j[z_{k+j}]$ and such that, for $k+1\le j \le k+n$, its residue at the poles~$z_j$ vanishes. $\cH_g(A;B)$ is a closed substack of $\cM_{g,k+n}$ of dimension 
\begin{gather}\label{eq:dimension of oH}
\dim\cH_g(A;B)=
\begin{cases}
2g-3+k,&\text{if $a_i<0$ for some $1\le i\le k$},\\
2g-2+k,&\text{otherwise}.
\end{cases}
\end{gather}
We denote by $\oH_g(A;B)$ the closure of $\cH_g(A;B)$ in~$\oM_{g,k+n}$.

\medskip

Suppose that $a_i=0$ for some $i$, say $i=1$. Then among the spaces $\oH_g(A;B)$ only the following ones are finite and non-empty: $\cH_0(0,a;-a-2)$, $a\ge 0$, and $\cH_0(0,-1,-1)$; and they are all just points.

\medskip

Suppose that $a_i\ne 0$ for all $1\le i\le k$. From formula~\eqref{eq:dimension of oH} it is easy to see that $\cH_g(A;B)$ can be non-empty and finite only in the case $g=0$ and if one of the following two conditions are satisfied:
\begin{enumerate}
\item $k=2$ and $a_1,a_2\ge 1$;

\smallskip

\item $k=3$ and exactly one of the three numbers $a_1,a_2,a_3$ is positive.
\end{enumerate}
Formally, there is also the case when $k=3$ and exactly one of the three numbers $a_1,a_2,a_3$ is negative, but this is obviously reduced to the first case. We see that these two cases correspond to differentials of the first and second type, respectively. So our goal is to compute the following two families of numbers:
\begin{align*}
&|\cH_0(a,b;-c_1,\ldots,-c_n)|,     && a,b\ge 1,\quad c_i\ge 2,\quad a+b-\sum c_i=-2,\\
&|\cH_0(a,-b,-c;-d_1,\ldots,-d_n)|, && a,b,c\ge 1,\quad d_i\ge 2,\quad a-b-c-\sum d_i=-2.
\end{align*}

\medskip

The space $\cH_0(A;B)$ can be described very explicitly. Indeed, given a $(k+n)$-tuple $Z=(z_1,\ldots,z_{k+n})$ of pairwise distinct points on $\CP^1$, there exists a unique, up to multiplication by a nonzero complex constant, meromophic differential~$\omega_{Z,A,B}$ on $\CP^1$ such that $(\omega_{Z,A,B})=\sum_{i=1}^k a_i[z_i]+\sum_{j=1}^n b_j[z_{j+k}]$. It is given by 
$$
\omega_{Z,A,B}=\prod_{i=1}^k(z-z_i)^{a_i}\prod_{j=1}^n (z-z_{j+k})^{b_j}dz,
$$
where, if $z_l=\infty$ for some $l$, then we assume that the corresponding factor is equal to $1$. Then  
\begin{gather*}
\cH_0(A;B)\cong\left.\left\{Z=(z_1,\ldots,z_{k+n})\in(\CP^1)^{k+n}\left|\begin{minipage}{5cm}\small $z_i\ne z_j$ and $\res_{z=z_l}\omega_{Z,A,B}=0$\\for $k+1\le l\le k+n$\end{minipage}\right.\right\}\right/{\PGL(2,\mbC)}.
\end{gather*}

\medskip

\begin{example}\label{example1}
Let us compute the number $|\cH_0(1,1;-2,-2)|$. A meromorphic differential~$\omega$ with $(\omega)=[1]+[t]-2[0]-2[\infty]$, $t\ne 0,1,\infty$, is given by
$$
\omega=\frac{(z-1)(z-t)}{z^2}dz,
$$
and we have $\res_{z=0}\omega=-1-t$, which is zero only for $t=-1$. Thus, $|\cH_0(1,1;-2,-2)|=1$.
\end{example}

\medskip

\begin{example}\label{example2}
Let us compute the number $|\cH_0(2,2;-3,-3)|$. A meromorphic differential~$\omega$ with $(\omega)=2[1]+2[t]-3[0]-3[\infty]$, $t\ne 0,1,\infty$, is given by
$$
\omega=\frac{(z-1)^2(z-t)^2}{z^3}dz,
$$
and we have $\res_{z=0}\omega=1+4t+t^2$. This quadratic polynomial has exactly two roots and, therefore, $|\cH_0(2,2;-3,-3)|=2$.
\end{example}

\medskip

\begin{example}\label{example3}
Let us compute the number $|\cH_0(2,2;-2,-2,-2)|$. A meromorphic differential~$\omega$ with $(\omega)=2[x]+2[y]-2[0]-2[1]-2[\infty]$, $x,y\ne 0,1,\infty$, $x\ne y$, is given by
$$
\omega=\frac{(z-x)^2(z-y)^2}{z^2(z-1)^2}dz,
$$
and we have 
$$
\res_{z=0}\omega=2xy(-x-y+xy),\qquad \res_{z=1}\omega=-2(-1+x+y-x^2y-xy^2+x^2y^2).
$$ 
So we have to find solutions of the system
$$
-x-y+xy=0=-1 + x + y - x^2 y - x y^2 + x^2 y^2,\quad x,y\ne 0,1,\quad x\ne y.
$$
From the first equation we obtain $x=\frac{y}{y-1}$, and substituting this in the second equation we get $\frac{y^2-y+1}{y-1}=0$, which has exactly two solutions. Thus, $|\cH_0(2,2;-2,-2,-2)|=2$.
\end{example}

\medskip

\begin{example}\label{example4}
Let us compute the number $|\cH_0(a+b-1,-1,-a;-b)|$, $a\ge 1$, $b\ge 2$. A meromorphic differential~$\omega$ with $(\omega)=(a+b-1)[t]-[1]-b[0]-a[\infty]$, $t\ne 0,1,\infty$, is given by
$$
\omega=\frac{(t-z)^{a+b-1}}{z^b(1-z)}dz,
$$
and we have 
$$
\res_{z=0}\omega=t^a\sum_{k=0}^{b-1}(-1)^k{a+b-1\choose k}t^{b-1-k}=:P(t).
$$
The ratio $\frac{P(t)}{t^a}$ is a polynomial of degree $b-1$, which doesn't vanish at $t=0$. Also $P(1)=(-1)^{b-1}{a+b-2\choose b-1}\ne 0$. From the elementary identity 
$$
P(t)-\frac{t-1}{a+b-1}P'(t)=(-1)^b{a+b-1\choose b}t^{a-1}
$$
we conclude that the polynomial $\frac{P(t)}{t^a}$ doesn't have multiple roots. Thus, $|\cH_0(a+b-1,-1,-a;-b)|=b-1$.
\end{example}

\medskip


\section{Differentials of the first type}

\subsection{Hurwitz numbers}

Let $\ov_1,\ldots,\ov_r$ be tuples of positive integers and $d\ge 1$. We denote~by 
$$
\Hur_d(\ov_1,\ldots,\ov_r)
$$
the number of ramified coverings $f\colon C\to\CP^1$, where $\deg f=d$, $C$ is a compact connected smooth algebraic curve, $f$ is ramified over $r$ fixed branch points in $\mbCP^1$, the ramification profile over the $i$-th branch point is given by the parts of~$\ov_i$ together with the necessary number of units, and the ramified covering $f$ is taken with the weight that is equal to the inverse of the order of the automorphism group of the covering, where we assume that an automorphism fixes the points corresponding to the parts of $\ov_i$. 

\medskip

We will be interested in the numbers $\Hur_{\sum c_i-n}((a+1),(b+1),(c_1-1,\ldots,c_n-1))$, where $a,b\ge 1$, $c_1,\ldots,c_n\ge 2$ and the condition $a+b-\sum c_i=-2$ is satisfied. Then the Riemann--Hurwitz formula says that for a corresponding ramified covering $f\colon C\to\CP^1$ we have $g(C)=0$.

\medskip

\subsection{The dispersionless KP hierarchy}\label{section:dKP}

Let $p$ and $f_i^{(j)}$, $i\ge 1$, $j\ge 0$, be formal variables, and consider the ring of polynomials $\cR_f:=\mbC\left[f_i^{(j)}\right]_{i\ge 1,\,j\ge 0}$. We also denote $f_i:=f_i^{(0)}$. Introduce a linear operator $\d_x\colon\cR_f\to\cR_f$ by $\d_x:=\sum_{i\ge 1,\,j\ge 0}f_i^{(j+1)}\frac{\d}{\d f_i^{(j)}}$. Let us endow the ring $\cR_f[p,p^{-1}]]$ with a Poisson structure by
$$
\{A,B\}:=\d_p A\cdot\d_x B-\d_p B\cdot\d_x A,\quad A,B\in\cR_f[p,p^{-1}]],
$$
where $\d_p:=\frac{\d}{\d p}$. For an element $A=\sum_{i\le m}a_i p^i\in\cR_f[p,p^{-1}]]$, denote $A_+:=\sum_{i=0}^m a_i p^i$.

\medskip

Let
\begin{gather}\label{eq:lambda(p)}
\lambda(p):=p+\sum_{i\ge 1}f_i p^{-i}\in\cR_f[p,p^{-1}]].
\end{gather}
The \emph{dispersionless KP hierarchy} is a system of evolutionary PDEs with dependent variables~$f_i$, $i\ge 1$, spatial variable $x$, and times $T_n$, $n\ge 1$, given by
$$
\frac{\d\lambda(p)}{\d T_n}=\left\{\left(\lambda(p)^n\right)_+,\lambda(p)\right\},\quad n\ge 1.
$$
For example, $\frac{\d f_i}{\d T_1}=f_i^{(1)}$ and 
$$
\frac{\d f_1}{\d T_2}=2 f_2^{(1)},\qquad \frac{\d f_2}{\d T_2}=2 f_3^{(1)}+2f_1f_1^{(1)}.
$$

\medskip

Define a change of variables $f_i\mapsto w_i(f_1,f_2,\ldots)$ by
$$
w_i(f_1,f_2,\ldots):=\res_{p=0}\lambda(p)^i,\quad i\ge 1.
$$
For example, we have
$$
w_1=f_1,\quad w_2=2f_2,\quad w_3=3f_3+3f_1^2,\quad w_4=4f_4+12f_1 f_2,\quad w_5=5f_5+20f_1f_3+10f_2^2+10f_1^3.
$$
Since 
\begin{multline*}
\frac{\d w_i}{\d T_j}=\res_{p=0}\left\{\lambda(p)^j_+,\lambda(p)^i\right\}=\res_{p=0}\left(\d_p\lambda(p)^j_+\d_x\lambda(p)^i-\d_x\lambda(p)^j_+\d_p\lambda(p)^i\right)=\\
=\res_{p=0}\left(-\lambda(p)^j_+\d_x\d_p\lambda(p)^i-\d_x\lambda(p)^j_+\d_p\lambda(p)^i\right)=-\d_x\res_{p=0}\left(\lambda(p)^j_+\d_p\lambda(p)^i\right),
\end{multline*}
we obtain that the dispersionless KP hierarchy written in the variables $w_i$ has the form
$$
\frac{\d w_i}{\d T_j}=\d_x R_{i,j},
$$
where $R_{i,j}$ are polynomials in $w_1,w_2,\ldots$, which can be found using the well-known formula (see e.g.~\cite[equations~(5.2.11)]{Tak14})
$$
R_{i,j}=-\res_{p=0}\left(\lambda(p)^j_+\d_p\lambda(p)^i\right).
$$
They satisfy the properties $R_{i,j}=R_{j,i}$ and $R_{i,1}=R_{1,i}=w_i$. The first nontrivial polynomials~are
$$
R_{2,2}=\frac{4}{3}w_3-2w_1^2,\qquad R_{2,3}=\frac{3}{2}w_4-3w_1w_2,\qquad R_{3,3}=\frac{9}{5}w_5-3w_1w_3-\frac{9}{4}w_2^2+3w_1^3.
$$

\medskip

It is easy to see that the coefficients of the inverse power series of~\eqref{eq:lambda(p)},
$$
p(\lambda)=\lambda+\sum_{\alpha\ge 1} t^{-\alpha-1}\lambda^{\alpha},
$$
are given by
$$
t^{-\alpha-1} = -\frac{w_{\alpha}}{\alpha}.
$$
The following result is well known, see e.g.~\cite[equation~(6.2.3)]{Tak14}, but for completeness we will give a short proof of it.

\medskip

\begin{lemma}
We have
\begin{gather}\label{eq:generating function for R}
\sum_{p,q\ge 1}\frac{1}{pq}R_{p,q}z^p\zeta^q=\log\left(1-\sum_{i\ge 1}\frac{1}{i}\frac{z^{i}-\zeta^{i}}{z^{-1}-\zeta^{-1}}w_i\right).
\end{gather}
\end{lemma}
\begin{proof}
First of all, for $\alpha,\beta\ge 0$, we note that
\begin{gather*}
\Coef_{\lambda_1^{-\alpha-1}\lambda_2^{-\beta-1}}\log\left(1-\sum_{i\ge 1}\frac{1}{i}\frac{\lambda_1^{-i}-\lambda_2^{-i}}{\lambda_1-\lambda_2}w_i\right)=\res_{\lambda_2=0}\res_{\lambda_1=0}\left[\lambda_1^{\alpha}\lambda_2^{\beta} \log \left(\frac{p(\lambda_2)-p(\lambda_1)}{\lambda_2-\lambda_1}\right)\right]
\end{gather*}
and then transform the right-hand side as follows:
\begin{align*}
&\res_{\lambda_2=0}\res_{\lambda_1=0}\left[\lambda_1^{\alpha}\lambda_2^{\beta} \left(\log \left(\frac{p(\lambda_2)-p(\lambda_1)}{p(\lambda_2)}\right)+\cancel{ \log\left(\frac{p(\lambda_2)}{\lambda_2-\lambda_1}\right)}\right)\right] =\\
=& \res_{p_2=0}\res_{p_1=0}\left[\frac{\d_{p_1}\lambda(p_1)^{\alpha+1}}{\alpha+1}\frac{\d_{p_2}\lambda(p_2)^{\beta+1}}{\beta+1} \log \left(\frac{p_2-p_1}{p_2}\right) \right] =\\
=& -\frac{\res_{p=0}\left(\lambda(p)^{\beta+1}_+\d_p\lambda(p)^{\alpha+1}\right)}{(\alpha+1)(\beta+1)}=\\
=&\frac{R_{\alpha+1,\beta+1}}{(\alpha+1)(\beta+1)},
\end{align*}
where the extra term in the first line can be eliminated since it is regular as $\lambda_1\to 0$.
\end{proof}

\medskip

\subsection{The representation theory of $\SL_2(\mbC)$}

Denote by $\rho_1$ the fundamental representation of the group $\SL_2=\SL_2(\mbC)$ and by $\rho_k$ its $k$-th symmetric power. The complete list of finite-dimensional irreducible representations of $\SL_2$ is $\{\rho_k\}_{k\ge 0}$. The tensor product $\rho_k\otimes\rho_l$ is decomposed in the sum of irreducible representations as follows:
$$
\rho_k\otimes\rho_l=\rho_{|k-l|}\oplus\rho_{|k-l|+2}\oplus\ldots\oplus\rho_{k+l}.
$$
We embed the group $\mbC^*$ into $\SL_2$ by
$$
\mbC^*\ni\lambda\mapsto
\begin{pmatrix}
\lambda & 0 \\
0 & \lambda^{-1}
\end{pmatrix}\in\SL_2.
$$
For a finite-dimensional $\SL_2$-representation $V$ and $d\in\mbZ$, we denote by $V_{[d]}\subset V$ the subspace of vectors having weight~$d$ with respect to the $\mbC^*$-action. So we have $V=\bigoplus_{d\in\mbZ}V_{[d]}$. 

\medskip

\subsection{The main result}

\begin{theorem}\label{theorem:first}
Let $a,b\ge 0$, $n\ge 1$, and $c_1,\ldots,c_n\ge 2$ be fixed integers satisfying the condition $a+b-\sum c_i=-2$. Then all the following numbers are equal:
\begin{enumerate}
\item $\displaystyle|\cH_0(a,b;-c_1,\ldots,-c_n)|$;

\smallskip

\item $\displaystyle\Hur_{\sum c_i-n}((a+1),(b+1),(c_1-1,\ldots,c_n-1))$;

\smallskip

\item $\displaystyle(-1)^{n+1}\frac{(c_1-1)\ldots(c_n-1)}{(a+1)(b+1)}\left.\frac{\d^n R_{a+1,b+1}}{\d w_{c_1-1}\ldots\d w_{c_n-1}}\right|_{w_*=0}$;

\smallskip

\item $\displaystyle(n-1)!\dim\left(\otimes_{i=1}^n\rho_{c_i-2}\right)_{[a-b]}$;

\smallskip

\item $\displaystyle(n-1)!\Coef_{t^{a+1}}\prod_{i=1}^n\frac{t-t^{c_i}}{1-t}$.
\end{enumerate}
\end{theorem}
\begin{proof}
To prove the equation (1)=(2), note that any residueless meromorphic differential~$\omega$ on~$\mbCP^1$ is exact and can be integrated to a meromorphic function~$f$. Moreover, if 
\begin{gather}\label{eq:divisor of omega,a-b case}
(\omega)=a[z_1]+b[z_2]-\sum_{i=1}^n c_i[z_{i+2}],
\end{gather}
where $z_1,\ldots,z_{n+2}\in\mbCP^1$ are pairwise distinct, then the only critical points of $f$ are the points $z_1,\ldots,z_{n+2}$ with multiplicities $a+1,b+1,c_1-1,\ldots,c_n-1$, respectively. Moreover, the set $\{z_3,\ldots,z_{n+2}\}$ is the set of poles of $f$, and the Riemann--Hurwitz formula implies that $f(z_1)\ne f(z_2)$. Since $\omega$ satisfying~\eqref{eq:divisor of omega,a-b case} is determined by the tuple $Z=(z_1,\ldots,z_{n+2})$ uniquely up to multiplication by a nonzero complex constant, a function $f$ satisfying $df=\omega$ is determined by~$Z$ uniquely up to the transformation $f\mapsto\alpha f+\beta$, $\alpha\in\mbC^*$, $\beta\in\mbC$. Let us fix the choice of $\alpha$ and $\beta$ by requiring that $f(z_1)=0$ and $f(z_2)=1$. This proves that the correspondence $f\mapsto df$ gives a bijection between the set of isomorphism classes of ramified coverings giving the number $\Hur_{\sum c_i-n}((a+1),(b+1),(c_1-1,\ldots,c_n-1))$ and the set $\cH_0(a,b;-c_1,\ldots,-c_n)$. Therefore, the equation (1)=(2) is proved.

\medskip

The equation (1)=(3) is the special case of~\cite[Theorem 3.5]{BRZ21} in genus~$0$. 

\medskip

Let us prove that (3)=(5). We have to prove that
$$
\frac{1}{pq}\left.\frac{\d^n R_{p,q}}{\d w_{l_1}\ldots\d w_{l_n}}\right|_{w_*=0}=-\frac{(n-1)!}{\prod_{i=1}^n l_i}\Coef_{z^p}\prod_{i=1}^n\frac{z^{l_i+1}-z}{1-z},\quad \begin{minipage}{4.5cm}$n,p,q,l_1,\ldots,l_n\ge 1$,\\$p+q=n+\sum l_i$.\end{minipage}
$$
Note that if we assign to $w_i$ degree $i+1$, then the polynomials~$R_{p,q}$ become homogeneous with $\deg R_{p,q}=p+q$. Transforming also $\prod_{i=1}^n\frac{z^{l_i+1}-z}{1-z}=\prod_{i=1}^n\frac{z^{l_i}-1}{z^{-1}-1}$, we see that the desired equation is equivalent to
$$
\frac{1}{pq}\left.\frac{\d^n R_{p,q}}{\d w_{l_1}\ldots\d w_{l_n}}\right|_{w_*=0}=-\frac{(n-1)!}{\prod_{i=1}^n l_i}\Coef_{z^p\zeta^q}\prod_{i=1}^n\frac{z^{l_i}-\zeta^{l_i}}{z^{-1}-\zeta^{-1}},\quad n,p,q,l_1,\ldots,l_n\ge 1,
$$
or, using the generating series, to
$$
\sum_{p,q\ge 1}\frac{1}{pq}\left.\frac{\d^n R_{p,q}}{\d w_{l_1}\ldots\d w_{l_n}}\right|_{w_*=0}z^p\zeta^q=-\frac{(n-1)!}{\prod_{i=1}^n l_i}\prod_{i=1}^n\frac{z^{l_i}-\zeta^{l_i}}{z^{-1}-\zeta^{-1}}.
$$
This can be further equivalently transformed as
$$
\left.\frac{\d^n}{\d w_{l_1}\ldots\d w_{l_n}}\left(\sum_{p,q\ge 1}\frac{1}{pq}R_{p,q}z^p\zeta^q\right)\right|_{w_*=0}=\left.\frac{\d^n}{\d w_{l_1}\ldots\d w_{l_n}}\log\left(1-\sum_{i\ge 1}\frac{1}{i}\frac{z^{i}-\zeta^{i}}{z^{-1}-\zeta^{-1}}w_i\right)\right|_{w_*=0},
$$
and therefore is equivalent to identity~\eqref{eq:generating function for R}.

\medskip

It remains to prove that (4)=(5). For a finite-dimensional representation~$V$ of $\SL_2$ denote 
$$
\chi_V(q):=\sum_{d\in\mbZ}q^d\dim V_{[d]}\in\mbZ[q,q^{-1}].
$$
Note that $\chi_{\rho_i}(q)=\frac{q^{i+1}-q^{-i-1}}{q-q^{-1}}$ and for two finite-dimensional $\SL_2$-representations $V$ and $W$ we have $\chi_{V\otimes W}(q)=\chi_V(q)\chi_W(q)$. So we compute
\begin{multline*}
\Coef_{t^{a+1}}\prod_{i=1}^n\frac{t^{c_i}-t}{t-1}=\Coef_{t^{a+1}}t^{\frac{\sum c_i}{2}}\prod_{i=1}^n\frac{t^{\frac{c_i-1}{2}}-t^{-\frac{c_i-1}{2}}}{t^{\frac{1}{2}}-t^{-\frac{1}{2}}}=\Coef_{t^{a+1}}t^{\frac{a+b+2}{2}}\prod_{i=1}^n\frac{t^{\frac{c_i-1}{2}}-t^{-\frac{c_i-1}{2}}}{t^{\frac{1}{2}}-t^{-\frac{1}{2}}}=\\
=\Coef_{t^{\frac{a-b}{2}}}\prod_{i=1}^n\frac{t^{\frac{c_i-1}{2}}-t^{-\frac{c_i-1}{2}}}{t^{\frac{1}{2}}-t^{-\frac{1}{2}}}\xlongequal{q=t^{\frac{1}{2}}}\Coef_{q^{a-b}}\prod_{i=1}^n\frac{q^{c_i-1}-q^{-c_i+1}}{q-q^{-1}}=\dim\left(\otimes_{i=1}^n\rho_{c_i-2}\right)_{[a-b]},
\end{multline*}
as required.
\end{proof}

\begin{remark}
The equality (2)=(5) was proved in~\cite[Proposition~2.1]{CMSZ20} using other methods.
\end{remark}

\begin{example}\label{example:n2-residue}
For $n=2$, the theorem gives
$$
|\cH_0(a,b;-c_1,-c_2)|=\min(a,b,c_1-1,c_2-1),\quad a,b\ge 0,\,c_1,c_2\ge 2,\,c_1+c_2=a+b+2.
$$
\end{example}

\medskip

\subsection{Dubrovin--Frobenius potential}

Consider the following formal power series in formal variables $t^\alpha$, $\alpha\in \mbZ^\star:=\mbZ\backslash\{-1\}$, collecting all the numbers $|\cH_0(a,b;-c_1,\ldots,-c_n)|$, $a,b\ge 0$, $c_1,\ldots,c_n\ge 2$, $a+b-\sum c_i=-2$, described above:
\begin{equation}\label{eq:Frobenius potential}
F(t^*):=\sum_{n\ge 1} \sum_{\substack{a,b\ge 0\\ c_1,\ldots,c_n\ge 2\\ a+b-\sum c_i=-2}} |\cH_0(a,b;-c_1,\ldots,-c_n)|\frac{t^a t^b}{2}\frac{t^{-c_1} \ldots t^{-c_n}}{n!}.
\end{equation}
Here and in what follows we use the subscript or superscript $*$ to denote all possible values of the corresponding index and we adhere to Einstein's convention of sum over repeated upper and lower indices.

\medskip

\begin{proposition}
Define the constant infinite matrix $\eta^{\alpha \beta}=\eta_{\alpha\beta}:= \delta_{\alpha+\beta,-2}$, $\alpha,\beta \in \mbZ^\star$ and the differential operators $E:=\sum_{\alpha\geq 0} t^\alpha\frac{\d}{\d t^\alpha}$ and $\widetilde{E}:=\sum_{\alpha\in \mbZ^\star} \alpha t^\alpha \frac{\d}{\d t^\alpha}$. Then the generating series~\eqref{eq:Frobenius potential} satisfies the following system of equations
\begin{align}
&\frac{\d^3 F}{\d t^\alpha \d t^\beta \d t^\mu} \eta^{\mu \nu} \frac{\d^3 F}{ \d t^\nu \d t^\gamma \d t^\delta}=\frac{\d^3 F}{\d t^\alpha \d t^\gamma \d t^\mu} \eta^{\mu \nu} \frac{\d^3 F}{ \d t^\nu \d t^\beta \d t^\delta}, && \alpha, \beta, \gamma, \delta \in \mbZ^\star,\label{eq:WDVVequations1}\\
&\frac{\d^3 F}{\d t^0 \d t^\alpha \d t^\beta}=\eta_{\alpha\beta}, && \alpha, \beta \in \mbZ^\star,\label{eq:WDVVequations2}\\
&E F=2F,\qquad  \widetilde{E}F= -2F.\label{eq:homogeneity}
\end{align}
\end{proposition}

\medskip

In other words, $F(t^*)$ is the Dubrovin--Frobenius potential for an infinite-dimensional Dubrovin--Frobenius manifold with metric $\eta=\eta_{\alpha\beta} dt^\alpha\otimes dt^\beta$, unit $e=\frac{\d}{\d t^0}$, homogeneous with respect to the two distinct Euler vector fields $E$ and $\widetilde{E}$, only the first of which is compatible with the unit in the sense that $[e,E]=e$. See \cite{Dub96} for the general theory of (finite-dimensional) Dubrovin--Frobenius manifolds. Infinite-dimensional Dubrovin--Frobenius manifolds have appeared in the literature, for instance in \cite{CDM11,Rai12,MWZ21}, but for the formal version needed here see the discussion in \cite[Section~4.1]{BR22} on tame infinite-rank partial cohomological field theories.

\medskip

\begin{proof}
In \cite[Proposition~1.8]{BRZ21} it was proved that, for $g,n\geq 0$ such that $2g-2+n>0$ and for integers $\alpha_1,\ldots,\alpha_n \in \mbZ^\star$, the fundamental classes of the closure~$\oH^{\res}_g(\alpha_1,\ldots,\alpha_n)$ in~$\oM_{g,n}$ of the loci~$\cH^{\res}_g(\alpha_1,\ldots,\alpha_n)$ in $\cM_{g,n}$ whose points correspond to genus $g$ smooth curves with marked points $z_1,\ldots,z_n$ such that there exists a residueless meromorphic differential whose divisor of zeros and poles is $\sum_{i=1}^n \alpha_i[z_i]$, which vanish unless $\sum_{i=1}^n \alpha_i = 2g-2$, form a tame partial cohomological field theory with phase space $V=\spn(e_\alpha)_{\alpha \in \mbZ^\star}$, metric $\eta(e_\alpha\otimes e_\beta)=\delta_{\alpha+\beta,-2}$, and unit $e_0$. This, in particular, entails the well-known fact that the genus $0$ primary potential~$F(t^*)$ of this partial CohFT satisfies the WDVV equations \eqref{eq:WDVVequations1}, \eqref{eq:WDVVequations2}.

\medskip

Homogeneity equations \eqref{eq:homogeneity} follow from the fact that, for $n\geq 3$, $\cH^{\res}_0(\alpha_1,\ldots,\alpha_n)$ is zero-dimensional only if exactly two $\alpha_i$ among the $\alpha_1,\ldots,\alpha_n \in \mbZ^\star$ are nonnegative, and empty unless $\sum_{i=1}^{n}\alpha_i = -2$.
\end{proof}

\medskip

Consider the following formal power series:
\begin{gather*}
\widetilde{p}(\lambda):=\sum_{\alpha\geq 0} t^\alpha \lambda^\alpha, \qquad \widetilde{P}(\lambda):=\sum_{\alpha \geq 0} t^\alpha \frac{\lambda^{\alpha+1}}{\alpha+1},
\end{gather*}
satisfying $\widetilde{P}'(\lambda) = \widetilde{p}(\lambda)$.

\medskip

\begin{proposition}
The Dubrovin--Frobenius potential \eqref{eq:Frobenius potential} can be written as
\begin{equation}\label{eq:Frobenius potential1}
\begin{split}
F(t^*) &= \res_{\lambda_2=0} \res_{\lambda_1=0} \left[ -\frac{1}{2} \widetilde{p}(\lambda_1) \widetilde{p}(\lambda_2) \log \left(\frac{p(\lambda_1)-p(\lambda_2)}{\lambda_1-\lambda_2} \right) \right]=\\
&=\res_{p=0} \frac{\widetilde{P}(\lambda(p))_+ \d_p \widetilde{P}(\lambda(p))}{2}.
\end{split}
\end{equation}
\end{proposition}
\begin{proof}
The right-hand side in the first line of \eqref{eq:Frobenius potential1} is clearly quadratic in the variables $t^\alpha$ with $\alpha\geq0$ and so is $F(t^*)$ by the homogeneity condition $EF=2F$, hence, in order to prove the first equality, it is enough to check that
$$
\frac{\d^2 F}{\d t^\alpha \d t^\beta} = \res_{\lambda_2=0}\res_{\lambda_1=0}\left[-\lambda_1^{\alpha}\lambda_2^{\beta} \log \left(\frac{p(\lambda_1)-p(\lambda_2)}{\lambda_1-\lambda_2}\right) \right], \quad \alpha,\beta\geq 0,
$$
which readily follows from equation \eqref{eq:generating function for R} and equality (1)=(3) in Theorem \ref{theorem:first}. For the second equality, we have
\begin{align*}
F(t^*) &= \res_{\lambda_2=0} \res_{\lambda_1=0} \left[ -\frac{1}{2}\widetilde{P}'(\lambda_1)\widetilde{P}'(\lambda_2) \log\left(\frac{p(\lambda_1)-p(\lambda_2)}{\lambda_1-\lambda_2}\right)\right]=\\
&= \res_{\lambda_2=0} \res_{\lambda_1=0} \left[ -\frac{1}{2}\widetilde{P}'(\lambda_1)\widetilde{P}'(\lambda_2) \left(\log\left(\frac{p(\lambda_1)-p(\lambda_2)}{p(\lambda_1)}\right)+\cancel{ \log\left(\frac{p(\lambda_1)}{\lambda_1-\lambda_2}\right)}\right)\right]=\\
&=\res_{p_2=0}\res_{p_1=0} \left[ -\frac{1}{2}\d_{p_1}\widetilde{P}(\lambda(p_1))\d_{p_2}\widetilde{P}(\lambda(p_2)) \log\left(\frac{p_1-p_2}{p_1}\right)\right] =\\
&= \res_{p=0} \frac{\widetilde{P}(\lambda(p))_+ \d_p \widetilde{P}(\lambda(p))}{2},
\end{align*}
where the extra term in the second line can be eliminated since it is regular as $\lambda_2\to 0$.
\end{proof}

\medskip

According to \cite{DZ01}, the primary flows of the principal hierarchy of the above Dubrovin--Frobenius manifold are computed as
\begin{equation}\label{eq:principal hierarchy}
\frac{\d t^{-\alpha-2}}{\d t^\beta_0} = \d_x \frac{\d^2 F(t^*)}{\d t^\alpha \d t^\beta}, \quad \alpha,\beta \in \mbZ^\star,
\end{equation}
where $t^\alpha=t^\alpha(x,t^*_*)$ is now seen as a formal loop on the Dubrovin--Frobenius manifold depending on time variables $t^\alpha_d$, $\alpha\in \mbZ^\star$, $d\geq 0$. Obviously, if $\alpha,\beta\ge 0$, then the right-hand side of~\eqref{eq:principal hierarchy} doesn't depend on $t^\gamma$ with $\gamma\ge 0$. This implies that restricting to $\alpha,\beta\ge 0$, we get a subsystem of the principal hierarchy, which by the equality (1)=(3) of Theorem~\ref{theorem:first} coincides with the dispersionless KP hierarchy. Thus, the above Dubrovin--Frobenius manifold is a natural Dubrovin--Frobenius manifold underlying the dispersionless KP hierarchy.

\medskip


\section{Differentials of the second type}

As remarked in the introduction, the following theorem was proved in \cite{CC19,GT22}. We provide an alternative proof in a different spirit, using Lemma \ref{lemma:WDVV}, a recursive relation between the number of differentials of the first and second type induced by the WDVV relations in the cohomology of the moduli space of rational stable curves.

\medskip

\begin{theorem}\label{theorem:second type}
We have
\begin{gather*}
|\cH_0(a,-b,-c;-d_1,\ldots,-d_n)|=n!\prod_{i=1}^n(d_i-1),
\end{gather*}
where $a\ge 0$, $b,c\ge 1$, $d_i\ge 2$, and $a-b-c-\sum d_i=-2$.
\end{theorem}
\begin{proof}
Let us introduce the following generating series:
\begin{align*}
&\theta^a_{b,c}(t_2,t_3,\ldots):=\sum_{n\ge 0}\frac{1}{n!}\sum_{\substack{d_1,\ldots,d_n\ge 2\\\sum d_i=a-b-c+2}}|\cH_0(a,-b,-c;-d_1,\ldots,-d_n)|t_{d_1}\cdots t_{d_n}, && a\ge 0,\,\, b,c\ge 1,\\
&P^{a,b}(t_2,t_3,\ldots):=\sum_{n\ge 1}\frac{1}{n!}\sum_{\substack{d_1,\ldots,d_n\ge 2\\\sum d_i=a+b+2}}|\cH_0(a,b;-d_1,\ldots,-d_n)|t_{d_1}\cdots t_{d_n}, && a,b\ge 0. 
\end{align*}
Clearly, we have
\begin{gather}\label{eq:elementary properties of theta}
\left\{
\begin{aligned}
&\theta^a_{b,c}=0\quad\text{if $b+c>a+2$},\\
&\theta^a_{b,c}=1\quad\text{if $b+c=a+2$},\\
&\theta^a_{b,c}=0\quad\text{if $b+c=a+1$}.
\end{aligned}
\right.
\end{gather}

\medskip

\begin{lemma}\label{lemma:WDVV}
For any $a,d,e,f\ge 1$ and $c\ge 2$, there is the following relation:
\begin{align}
\sum_{b\ge 2}\frac{\d^2 P^{a,d}}{\d t_c\d t_b}\theta^{b-2}_{e,f}+\sum_{b\ge 2}\frac{\d P^{a,b-2}}{\d t_c}\frac{\d\theta^d_{e,f}}{\d t_b}=&\sum_{b\ge 2}\theta^a_{f,b}\frac{\d P^{b-2,d}}{\d t_e\d t_c}+\theta^a_{f,1}\frac{\d\theta^d_{1,e}}{\d t_c}+\sum_{b\ge 2}\frac{\d P^{a,b-2}}{\d t_f}\frac{\d\theta^d_{b,e}}{\d t_c}\label{eq:WDVV2}\\
&+\sum_{b_1,b_2\ge 2}(b_2-1)\theta^a_{f,b_1}\theta^d_{e,b_2}\frac{\d P^{b_1-2,b_2-2}}{\d t_c}\notag\\
&+\sum_{b_1,b_2\ge 2}(b_1+b_2-2)\theta^a_{f,b_1}\frac{\d\theta^d_{e,b_2}}{\d t_c}P^{b_1-2,b_2-2}\notag.
\end{align}
\end{lemma}
\begin{proof}
See Section~\ref{section:proof of lemma}.
\end{proof}

\medskip

Consider relation~\eqref{eq:WDVV2} with $d=3$ and $c=2$. Clearly,
\begin{gather*}
\frac{\d^2 P^{a,3}}{\d t_2\d t_b}=0,\quad\text{if $b>a+3$},
\end{gather*}
and from Example~\ref{example:n2-residue} it follows that
\begin{gather*}
\frac{\d^2 P^{a,3}}{\d t_2\d t_{a+3}}=1.
\end{gather*}
We see that the coefficient of $\theta^{b-2}_{e,f}$ in the first sum on the left-hand side of equation~\eqref{eq:WDVV2} is one for $b=a+3$ and is zero for $b>a+3$. Looking also at the other terms in~\eqref{eq:WDVV2}, we notice that if $a\ge 3$, then relation~\eqref{eq:WDVV2} allows to express $\theta^{a+1}_{e,f}$ in terms of $\theta^p_{q,r}$ with $p\le a$. So relation~\eqref{eq:WDVV2} allows to express all polynomials $\theta^a_{e,f}$ in terms of the polynomials 
$\theta^p_{q,r}$ with $p\le 3$, but all these polynomials are determined by properties~\eqref{eq:elementary properties of theta} and Example~\ref{example4}.

\medskip

So it remains to check the following statement.

\begin{lemma}
The polynomials $\widetilde\theta^a_{b,c}$, $a\ge 0$, $b,c\ge 1$, defined by
$$
\widetilde\theta^a_{b,c}:=\sum_{n\ge 0}\sum_{\substack{d_1,\ldots,d_n\ge 2\\\sum d_i=a+2-b-c}}\prod_{i=1}^n(d_i-1)t_{d_i}
$$
satisfy equation~\eqref{eq:WDVV2}.
\end{lemma}
\begin{proof}
Let us substitute the polynomials $\ttheta^a_{b,c}$ in~\eqref{eq:WDVV2} and sum both sides over $a,d\ge 0$ with the coefficient $x^{a+1}y^{d+1}$. It is sufficient to prove that we obtain an equality of two formal power series in $x,y$, and $t_i$. 

\medskip

Introduce the following formal power series:
\begin{gather*}
A(x,y):=1+\sum_{i\ge 0}\frac{x^{i}-y^i}{x^{-1}-y^{-1}}t_{i+1},\qquad B(x):=1-\sum_{i\ge 1}(i-1)x^i t_i.
\end{gather*}
The equality (1)=(3) of Theorem~\ref{theorem:first} together with formula~\eqref{eq:generating function for R} imply that
$$
\sum_{a,b\ge 0}P^{a,b}x^{a+1}y^{b+1}=-\log A(x,y).
$$
It is easy to see that 
$$
\ttheta^a_{b,c}=\Coef_{z^{a+2-b-c}}\frac{1}{B(z)}.
$$

\medskip

We compute
\begin{flalign*}
E_1:=\sum_{a,d\ge 0}x^{a+1}y^{d+1}\sum_{b\ge 2}\frac{\d^2 P^{a,d}}{\d t_c\d t_b}\ttheta^{b-2}_{e,f}=&\frac{x^{c-1}-y^{c-1}}{x^{-1}-y^{-1}}\frac{1}{A(x,y)^2}\sum_{b\ge 2}\frac{x^{b-1}-y^{b-1}}{x^{-1}-y^{-1}}\Coef_{z^{b-e-f}}\frac{1}{B(z)}=&&\\
=&\frac{x^{c-1}-y^{c-1}}{x^{-1}-y^{-1}}\frac{1}{A(x,y)^2}\sum_{b\ge 2}\frac{x^{b-1}-y^{b-1}}{x^{-1}-y^{-1}}\Coef_{z^{b-1}}\frac{z^{e+f-1}}{B(z)}=&&\\
=&\boxed{\frac{x^{c-1}-y^{c-1}}{x^{-1}-y^{-1}}\frac{1}{A(x,y)^2}\left(\frac{x^{e+f-1}}{B(x)}-\frac{y^{e+f-1}}{B(y)}\right)\frac{1}{x^{-1}-y^{-1}}};&&
\end{flalign*}
\begin{flalign*}
E_2:=&\sum_{\substack{a,d\ge 0\\b\ge 2}}x^{a+1}y^{d+1}\frac{\d P^{a,b-2}}{\d t_c}\frac{\d\ttheta^d_{e,f}}{\d t_b}=&&\\
=&-\sum_{\substack{d\ge 0\\b\ge 2}}y^{d+1}\Coef_{y^{b-1}}\left(\frac{x^{c-1}-y^{c-1}}{x^{-1}-y^{-1}}\frac{1}{A(x,y)}\right)\Coef_{z^{d+2-e-f}}\frac{(b-1)z^b}{B(z)^2}=&&\\
=&-\Coef_{y^{b-1}}\left(\frac{x^{c-1}-y^{c-1}}{x^{-1}-y^{-1}}\frac{1}{A(x,y)}\right)\frac{(b-1)y^{b+e+f-1}}{B(y)^2}=&&\\
=&\boxed{-y\d_y\left(\frac{x^{c-1}-y^{c-1}}{x^{-1}-y^{-1}}\frac{1}{A(x,y)}\right)\frac{y^{e+f}}{B(y)^2}};&&
\end{flalign*}
\begin{flalign*}
E_3:=&\sum_{a,d\ge 0}x^{a+1}y^{d+1}\sum_{b\ge 2}\theta^a_{f,b}\frac{\d P^{b-2,d}}{\d t_e\d t_c}=&&\\
=&\sum_{b\ge 2}\frac{x^{f+b-1}}{B(x)}\Coef_{x^{b-1}}\left(\frac{x^{e-1}-y^{e-1}}{x^{-1}-y^{-1}}\frac{x^{c-1}-y^{c-1}}{x^{-1}-y^{-1}}\frac{1}{A(x,y)^2}\right)=&&\\
=&\boxed{\frac{x^{f}}{B(x)}\frac{x^{e-1}-y^{e-1}}{x^{-1}-y^{-1}}\frac{x^{c-1}-y^{c-1}}{x^{-1}-y^{-1}}\frac{1}{A(x,y)^2}};&&
\end{flalign*}
\begin{flalign*}
E_4:=\sum_{a,d\ge 0}x^{a+1}y^{d+1}\theta^a_{f,1}\frac{\d\theta^d_{1,e}}{\d t_c}=\boxed{\frac{(c-1)x^f y^{e+c}}{B(x)B(y)^2}};&&
\end{flalign*}
\begin{flalign*}
E_5:=&\sum_{a,d\ge 0}x^{a+1}y^{d+1}\sum_{b\ge 2}\frac{\d P^{a,b-2}}{\d t_f}\frac{\d\theta^d_{b,e}}{\d t_c}=&&\\
=&-\sum_{b\ge 2}\Coef_{y^{b-1}}\left(\frac{x^{f-1}-y^{f-1}}{x^{-1}-y^{-1}}\frac{1}{A(x,y)}\right)\frac{(c-1)y^{b+e+c-1}}{B(y)^2}=&&\\
=&\boxed{-\frac{x^{f-1}-y^{f-1}}{x^{-1}-y^{-1}}\frac{1}{A(x,y)}\frac{(c-1)y^{e+c}}{B(y)^2}};&&
\end{flalign*}
\begin{flalign*}
E_6:=&\sum_{a,d\ge 0}x^{a+1}y^{d+1}\sum_{b_1,b_2\ge 2}(b_2-1)\theta^a_{f,b_1}\theta^d_{e,b_2}\frac{\d P^{b_1-2,b_2-2}}{\d t_c}=&&\\
=&-\sum_{b_1,b_2\ge 2}(b_2-1)\frac{x^{f+b_1-1}}{B(x)}\frac{y^{e+b_2-1}}{B(y)}\Coef_{x^{b_1-1}y^{b_2-1}}\left(\frac{x^{c-1}-y^{c-1}}{x^{-1}-y^{-1}}\frac{1}{A(x,y)}\right)=&&\\
=&\boxed{-\frac{x^{f}}{B(x)}\frac{y^{e}}{B(y)}y\d_y\left(\frac{x^{c-1}-y^{c-1}}{x^{-1}-y^{-1}}\frac{1}{A(x,y)}\right)};&&
\end{flalign*}
\begin{flalign*}
E_7:=&\sum_{a,d\ge 0}x^{a+1}y^{d+1}\sum_{b_1,b_2\ge 2}(b_1+b_2-2)\theta^a_{f,b_1}\frac{\d\theta^d_{e,b_2}}{\d t_c}P^{b_1-2,b_2-2}=&&\\
=&-\sum_{b_1,b_2\ge 2}(b_1+b_2-2)\frac{x^{f+b_1-1}}{B(x)}\frac{(c-1)y^{c+e+b_2-1}}{B(y)^2}\Coef_{x^{b_1-1}y^{b_2-1}}\log A(x,y)=&&\\
=&\boxed{-\frac{x^{f}}{B(x)}\frac{(c-1)y^{e+c}}{B(y)^2}\left(x\d_x+y\d_y\right)\log A(x,y)}.&&
\end{flalign*}

\medskip

Let $\mu_a:=x^{a-1}-y^{a-1}$ and $\tA(x,y):=(x^{-1}-y^{-1})A(x,y)$. We compute
\begin{align*}
&E_1-E_3=\frac{\mu_c y^e}{\tA(x,y)^2}\left(\frac{x^f y^{-1}}{B(x)}-\frac{y^{f-1}}{B(y)}\right),\\
&E_2-E_6=y\d_y\left(\frac{\mu_c}{\tA(x,y)}\right)\frac{y^e}{B(y)}\left(\frac{x^f}{B(x)}-\frac{y^f}{B(y)}\right).
\end{align*}
We have to check the vanishing of the expression
\begin{align}
&y^{-e}(E_1+E_2-E_3-E_4-E_5-E_6-E_7)=\label{eq:seven terms}\\
=&\frac{\mu_c}{\tA(x,y)^2}\left(\frac{x^f y^{-1}}{B(x)}-\underline{\frac{y^{f-1}}{B(y)}}\right)+y\d_y\left(\frac{\mu_c}{\tA(x,y)}\right)\frac{1}{B(y)}\left(\frac{x^f}{B(x)}-\underline{\frac{y^f}{B(y)}}\right)-\frac{(c-1)x^f y^{c}}{B(x)B(y)^2}\notag\\
&+\left(x^{f-1}-\underline{y^{f-1}}\right)\frac{1}{\tA(x,y)}\frac{(c-1)y^{c}}{B(y)^2}+\frac{x^{f}}{B(x)}\frac{(c-1)y^{c}}{B(y)^2}\left(x\d_x+y\d_y\right)\log A(x,y).\notag
\end{align}
Note that
\begin{gather}\label{eq:dtA-and-B}
y^2\d_y\tA(x,y)=B(y).
\end{gather}
Let us collect the underlined terms on the right-hand side of~\eqref{eq:seven terms}:
$$
-\frac{\mu_c y^{f-1}}{\tA(x,y)^2 B(y)}-\frac{(y\d_y\mu_c)y^f}{\tA(x,y)B(y)^2}+\frac{\mu_c\left(y^2\d_y\tA(x,y)\right)y^{f-1}}{\tA(x,y)^2B(y)^2}-\frac{(c-1)y^{c+f-1}}{\tA(x,y)B(y)^2}.
$$
The second and the fourth terms here obviously cancel each other, while the first and the third terms cancel each other by~\eqref{eq:dtA-and-B}.

\medskip

Thus, the expression on the right-hand side of~\eqref{eq:seven terms} is equal to
\begin{align*}
&\frac{\mu_c}{\tA(x,y)^2}\frac{x^fy^{-1}}{B(x)}+y\d_y\left(\frac{\mu_c}{\tA(x,y)}\right)\frac{x^f}{B(x)B(y)}-\frac{(c-1)x^f y^{c}}{B(x)B(y)^2}\\
&+\frac{x^{f-1}}{\tA(x,y)}\frac{(c-1)y^{c}}{B(y)^2}+\frac{x^{f}}{B(x)}\frac{(c-1)y^{c}}{B(y)^2}\left(x\d_x+y\d_y\right)\log A(x,y)=\\
=&\cancel{\frac{\mu_c}{\tA(x,y)^2}\frac{x^f y^{-1}}{B(x)}}-\frac{(c-1)y^{c-1}}{\tA(x,y)}\frac{x^f}{B(x)B(y)}-\cancel{\frac{\mu_c y^{-1}}{\tA(x,y)^2}\frac{x^f}{B(x)}}-\underline{\frac{(c-1)x^f y^{c}}{B(x)B(y)^2}}\\
&+\frac{x^{f-1}}{\tA(x,y)}\frac{(c-1)y^{c}}{B(y)^2}+\underline{\frac{x^{f}}{B(x)}\frac{(c-1)y^{c}}{B(y)^2}\frac{\sum_{i\ge 0}(i+1)\frac{x^{i}-y^i}{x^{-1}-y^{-1}}t^{i+1}}{A(x,y)}}=\\
=&-\frac{(c-1)y^{c-1}}{\tA(x,y)}\frac{x^f}{B(x)B(y)}+\frac{x^{f-1}}{\tA(x,y)}\frac{(c-1)y^{c}}{B(y)^2}+\frac{x^{f}}{B(x)}\frac{(c-1)y^{c}}{B(y)^2}\frac{\sum_{i\ge 0}i\frac{x^{i}-y^i}{x^{-1}-y^{-1}}t^{i+1}-1}{A(x,y)}=\\
=&-\frac{(c-1)y^{c-1}}{\tA(x,y)}\frac{x^f}{B(x)B(y)}+\frac{x^{f-1}}{\tA(x,y)}\frac{(c-1)y^{c}}{B(y)^2}+\frac{x^{f}}{B(x)}\frac{(c-1)y^{c}}{B(y)^2}\frac{\frac{B(y)}{y}-\frac{B(x)}{x}}{\tA(x,y)}=\\
=&0,
\end{align*}
as required.
\end{proof}
\end{proof}

\medskip

\section{Proof of Lemma \ref{lemma:WDVV}}\label{section:proof of lemma}

\subsection{Multiscale differentials with residue conditions}\label{section:multiscale} For $A=(a_1,\ldots,a_k) \in \mbZ^k$ and $B=(b_1,\ldots,b_n) \in \mbZ^n$ with $b_j\le -2$, $1 \le j\le n$, let us briefly review the properties of the moduli space $\oH_g (A;B)$ from the point of view of multiscale differentials with residue conditions as treated in \cite{CMZ20}.

\medskip

In \cite[Sections~3 and~4.1]{CMZ20} (see also~\cite[Section~2]{BCGGM19}) the authors construct a proper smooth Deligne--Mumford stack $\oB_g(A;B)$ as a moduli stack for families of equivalence classes of projectivized multiscale differentials with residue conditions on stable curves, whose definition we will recall in this section. Our motivation comes from the fact that this moduli stack comes with a forgetful map $p\colon\oB_g(A;B)\to \oM_{g,k+n}$ associating to a projectivized multiscale differential on a stable marked curve $C$ the stable marked curve itself and that this map restricts to an isomorphism of Deligne--Mumford stacks $p:B_g(A;B)\to\cH_g(A;B)$ on the open substack $B_g(A;B) = p^{-1}(\cM_{g,k+n})$ of projectivized multiscale differentials on smooth curves, so that 
$$
[\oH_g(A;B)]= p_*[\oB_g(A;B)].
$$
Moreover, the boundary $\oB_g(A;B)\setminus B_g(A;B)$ is a normal crossing divisor and \cite{CMZ20} gives a modular description of these boundary strata. Crucially, the map $p$ is compatible with the stratified structures of the two spaces and we will use this fact to understand the intersection of $[\oH_g (A;B)]$ with the boundary strata of $\oM_{g,k+n}$, in particular those formed by stable curves with one separating node.

\medskip

In what follows, given a stable curve $C$ with associated stable graph $\Gamma_C$, we will denote its irreducible components by $C_v$ for $v\in V(\Gamma_C)$ and we will use the same notation for the marked points of $C$ and the corresponding legs of the associated stable graph $\Gamma_C$, for nodes of $C$ and the corresponding edges of $\Gamma_C$, and for branches of nodes on irreducible components $C_v$ of $C$ and the corresponding half-edges of $\Gamma_C$. Given a leg $x_i\in L(\Gamma_C)$ or a half-edge $h\in H(\Gamma_C)$, we denote by $v(x_i)$ or $v(h)$ the vertex to which they are attached.

\medskip

Firstly, an \emph{enhanced level graph} is a stable graph $\Gamma$ of genus $g$ with a set $L(\Gamma)$ of $n$ marked legs together with:
\begin{enumerate}
\item a total preorder\footnote{A preorder relation $\leq$ is reflexive and transitive, but $x \leq y$ and $y \leq x$ do not necessarily imply $x=y$.} on the set $V(\Gamma)$ of vertices. We describe this preorder by a surjective level function $\ell\colon V(\Gamma)\to \{0,-1,\ldots,-L\}$. An edge is called \emph{horizontal} if it is attached to vertices on the same level and \emph{vertical} otherwise.

\smallskip

\item a function $\kappa\colon E(\Gamma)\to \mbZ_{\geq 0}$ assigning a nonnegative integer~$\kappa_e$ to each edge $e\in E(\Gamma)$, such that $\kappa_e=0$ if and only if $e$ is horizontal.
\end{enumerate}
For every level $0\leq j\leq -L$, let $C_{(j)}$ be the (possibly disconnected) stable curve obtained from~$C$ by removing all irreducible components whose level is not $j$ and let $C_{(>j)}$  be the (possibly disconnected) stable curve obtained from~$C$ by removing all irreducible components whose level is smaller than or equal to~$j$.

\medskip

Secondly, given a meromorphic differential~$\omega$ on a smooth curve $C$ and a point $p\in C$, if~$\omega$ has order $\ord_p \omega = a\neq -1$ at $p$ then for a local coordinate $z$ in a neighborhood of $p$ such that $z(p)=0$ we have, locally, $\omega = (cz^a + O(z^{a+1}))dz$ for some $c\in \mbC^*$. Then the $k=|a+1|$ roots $\zeta$ such that $\zeta^{a+1}=c^{-1}$ determine $k$ projectivized vectors $\left.\zeta\frac{\d}{\d z}\right|_p\in T_pC/\mbR_{>0}$ (if $a\geq 0$) or $\left.-\zeta\frac{\d}{\d z}\right|_p\in T_pC/\mbR_{>0}$ (if $a<-1$) which are called \emph{outgoing} or \emph{incoming prongs} of~$\omega$, respectively. The set of outgoing (resp. incoming) prongs at $p$ is denoted by $P^\mathrm{out}_p$ (resp. $P^\mathrm{in}_p$).

\medskip

Thirdly, let $A=(a_1,\ldots,a_k) \in \mbZ^k$ and $B=(b_1,\ldots,b_n) \in \mbZ^n$ with $b_j\le -2$, $1 \le j\le n$. Then a \emph{multiscale differential} of profile $(A;B)$, with $\sum_{i=1}^k a_i+\sum_{j=1}^k b_j= 2g-2$, on a stable curve~$C$ of genus $g$ with $k+n$ marked points $x_1,\ldots,x_{k+n}\in C$, with \emph{zero residues} at $x_{k+1},\ldots,x_{k+n}$ consists of:
\begin{enumerate}
\item a structure of enhanced level graph $(\Gamma_C, \ell, \kappa)$ on the dual graph $\Gamma_C$ of $C$ (where a node is said to be vertical or horizontal if the corresponding edge is);

\smallskip

\item a collection of meromorphic differentials $\omega_v$, one on each irreducible component~$C_v$ of~$C$, $v\in V(\Gamma_C)$, holomorphic and non-vanishing outside of marked points and nodes, such that the following conditions are satisfied:
\begin{itemize}
\item[(i)] $\ord_{x_i}\omega_{v(x_i)} = a_i$ for $1\leq i \leq k$, and $\ord_{x_j}\omega_{v(x_j)} = b_{j-k}$ for $k+1\leq j \leq k+n$.

\smallskip

\item[(ii)] $\res_{x_j} \omega_{v(x_j)}=0$, $k+1\leq j \leq k+n$. 

\smallskip

\item[(iii)] If $q_1\in C_{v_1}$ and $q_2 \in C_{v_2}$, $v_1,v_2\in V(\Gamma_C)$, form a node $e\in E(\Gamma_C)$, then
\begin{equation*}
\ord_{q_1}\omega_{v_1}+\ord_{q_2}\omega_{v_2}=-2.
\end{equation*}

\smallskip

\item[(iv)] If $q_1\in C_{v_1}$ and $q_2 \in C_{v_2}$, $v_1,v_2\in V(\Gamma_C)$, form a node $e\in E(\Gamma_C)$, then $\ell(v_1)\ge \ell(v_2)$ if and only if $\ord_{q_1}\omega_{v_1}\ge -1$. Together with the previous property, this implies that $\ell(v_1)=\ell(v_2)$ if and only if $\ord_{q_1}\omega_{v_1}=-1$.

\smallskip

\item[(v)] If $q_1\in C_{v_1}$ and $q_2 \in C_{v_2}$, $v_1,v_2\in V(\Gamma_C)$, form a horizontal node $e\in E(\Gamma_C)$ (i.e. $\kappa_e=0$), then
\begin{equation}\label{eq:residue condition at horizontal nodes}
\res_{q_1} \omega_{v_1}+\res_{q_2} \omega_{v_2}=0.
\end{equation}

\smallskip

\item[(vi)] For every level $-1\leq l\leq -L$ of $\Gamma_C$ and for every connected component $Y$ of $C_{(>l)}$ such that $Y$ does not contain any marked pole $x_i$, with $a_i<0$ and $1 \le i \le k$,
\begin{equation}\label{eq:residue condition at vertical nodes}
\sum_{q\in Y\cap C_{(l)}} \res_{q^-} \omega_{v(q^-)} =0,
\end{equation}
where $q^+\in Y$ and $q^-\in C_{(l)}$ form the vertical node $q \in Y\cap C_{(l)}$.
\end{itemize}

\smallskip

\item a cyclic order-reversing bijection $\sigma_q\colon P^\mathrm{out}_{q^+} \to P^\mathrm{in}_{q^-}$ for each vertical node~$q$ formed by identifying $q^+$ on the upper level with $q^-$ on the lower level, where $\kappa_q=|P^\mathrm{out}_{q^+}|=|P^\mathrm{in}_{q^-}|$.
\end{enumerate}

\medskip

Lastly, there is an action of the universal cover of the torus $\mbC^{L}\to (\mbC^*)^{L}$ on multiscale differentials with residue conditions by rescaling the differentials with strictly negative levels and rotating the prong matchings between levels accordingly, producing fractional Dehn twists. The stabilizer of this action is called the \emph{twist group} of the enhanced level graph and denoted by~$\mathrm{Tw}_\Gamma$. Two multiscale differentials with residue conditions are defined to be equivalent if they differ by the action of $T_\Gamma:=\mbC^{L}/\mathrm{Tw}_\Gamma$. By further quotienting by the action of $\mbC^*$ rescaling the differentials on all levels and leaving all prong-matchings untouched, we obtain equivalence classes of projectivized multiscale differentials with residue conditions.

\medskip

\begin{remark}\label{remark:R-space}
Using notation from~\cite[Section~4.1]{CMZ20}, condition (2)(vi) is a reformulation of the $\mathfrak{R}$-global residue condition in the particular case when $\lambda$ is the partition of $H_p$ in one-element subsets and $\lambda_{\mathfrak{R}}$ is the set of parts of $\lambda$ corresponding to the residueless poles. This condition is understood by realizing that the residues appearing in the sum \eqref{eq:residue condition at vertical nodes} correspond to periods around the ``waist'' of an undegeneration of the corresponding node (see below for the explicit form of this undegeneration). The sum of such periods has to equal the sum of residues in $Y$ by the residue theorem applied to such undegeneration.
\end{remark}

\medskip

As a special case of \cite[Proposition 4.2]{CMZ20} (corresponding to the choice of $\lambda$ and $\lambda_\mathfrak{R}$ described in Remark~\ref{remark:R-space}), we have the following result.

\begin{proposition}\cite{CMZ20}\label{proposition:moduli of multiscale}
\begin{enumerate}[ 1.]

\item Given $A=(a_1,\ldots,a_k) \in \mbZ^k$ and $B=(b_1,\ldots,b_n) \in \mbZ^n$ with $b_j\le -2$, $1 \le j\le n$, there is a proper smooth Deligne--Mumford stack $\oB_g(A;B)$ containing $B_g(A;B)$ as an open dense substack whose complement is a normal crossing divisor. $\oB_g(A;B)$ is a moduli stack for families of equivalence classes of projectivized multiscale differentials with residue conditions. Its dimension is
$$
\dim \oB_g(A;B)= 
\begin{cases}
2g-2+k,&\text{if $a_i\ge 0$ for all $1\le i\le k$},\\
2g-3+k,&\text{otherwise}.
\end{cases} 
$$

\smallskip

\item We denote the closure of the stratum parameterizing multiscale differentials whose enhanced level graph is $(\Gamma,\ell,\kappa)$ by $D_{(\Gamma,\ell,\kappa)}$ or simply by $D_\Gamma$. Then $D_\Gamma$ is a proper smooth closed substack of $\oB_g(A;B)$ of codimension
$$
\codim D_{\Gamma} = h+L,
$$
where~$h$ is the number of horizontal edges in $(\Gamma,\ell,\kappa)$ and~$L+1$ is the number of levels.
\end{enumerate}
\end{proposition}

\medskip

\begin{remark}
Notice that the multiscale differentials appearing at level $l$ in the generic boundary stratum of $\oB_g(A;B)$ are of a new type, because of the global residue condition (2)(vi) not only is more general than just requiring the vanishing of each residue at a subset of the marked poles, but also involves poles on different connected components of the curve $C_{(l)}$. This is the reason why, in \cite{CMZ20}, the authors consider moduli stacks of more general projectivized multiscale differentials with residue conditions, where the underlying stable curve can be disconnected and the residue condition constrains sums of residues at fixed disjoint subsets of poles. With their choice, the moduli spaces involved in the boundary strata of another moduli space are all of the same type.

\medskip

This problem with the moduli spaces $\oB_g(A;B)$ we introduced is less severe whenever one is only interested in boundary strata where the underlying stable curve has only separating nodes (i.e. it is of compact type). Indeed, in that case, each sum in condition (2)(vi) can only involve at most one pole for each of the $N_l$ connected component of the curve $C_{(l)}$  at level $l$, and the effect of all the sums with more than one summand is simply to reduce the $(\mbC^*)^{N_l-1}$-symmetry consisting in rescaling the differentials on each connected component by relative multiplicative constants.

\medskip

This happens, for instance, if one studies genus $0$ multiscale differentials (as we do in this paper) or if one is interested in the intersection of the boundary strata with the pull-back of the top Chern class $\lambda_g$ of the Hodge bundle from $\oM_{g,k+n}$ (since $\lambda_g$ is well known to vanish on the locus of curves of non-compact type) or, more in general, in the context of partial cohomological field theories.
\end{remark}

\medskip

\subsection{Proof of Lemma \ref{lemma:WDVV}} 

For $n\geq 0$ and for integers $a,d\ge 0$, $e,f>0$, $c,c_1\ldots,c_n\ge 2$, consider the closed substack of meoromorphic differentials 
$$
\oH_0(a,d,-e,-f;-c,-c_1,\ldots,-c_n)\subset \oM_{0,5+n}.
$$
We will denote $(-c_1,\ldots,-c_n)=\colon -C$ and more in general $(-c_{i_1},\ldots,-c_{i_r})=\colon -C_{I}$ for any $I=\{i_1,\ldots,i_r\}\subset \{1,\ldots,n\}$, so that $\oH_0(a,d,-e,-f;-c,-c_1,\ldots,-c_n)=\oH_0(a,d,-e,-f;-c,-C)$.  Consider moreover the moduli stack $\oB_0(a,d,-e,-f;-c,-C)$ of projectivized multiscale differentials described above, with its natural projection
\begin{equation}\label{eq:projection}
p\colon\oB_0(a,d,-e,-f;-c,-C)\to \oH_0(a,d,-e,-f;-c,-C) \subset \oM_{0,5+n},
\end{equation}
which restricts to an isomorphism $p\colon B_0(a,d,-e,-f;-c,-C)\to \cH_0(a,d,-e,-f;-c,-C)$, so that $[\oH_0(a,d,-e,-f;-c,-C)]=p_*[\oB_0(a,d,-e,-f;-c,-C)]$. By Proposition \ref{proposition:moduli of multiscale}, we have $\dim \oH_0(a,d,-e,-f;-c,-C)=\dim \oB_0(a,d,-e,-f;-c,-C)=1$.

\medskip

We want to intersect $[\oH_0(a,d,-e,-f;-c,-C)]$ with the pull-back to $H^2(\oM_{0,5+n})$ of the WDVV relation in $H^2(\oM_{0,4})$:
$$
\tikz[baseline=-1mm]{\draw (A)--(B);\legm{A}{120}{a};\legm{A}{-120}{-c};\legm{B}{-60}{-f};\legm{B}{60}{-e};\gg{0}{A};\gg{0}{B};}=\tikz[baseline=-1mm]{\draw (A)--(B);\legm{A}{120}{a};\legm{A}{-120}{-f};\legm{B}{-60}{-c};\legm{B}{60}{-e};\gg{0}{A};\gg{0}{B};} \in H^2(\oM_{0,4}),
$$
via the map that forgets the $n+1$ marked points with multiplicities $d$ and $-C$. Notice here that we are slightly abusing the usual stable graph notation indicating the classes of boundary strata in $\oM_{0,4}$, since we are using the zero and pole multiplicities to label the marked points of curves that, in general, do not belong to $\oH_0(a,d,-e,-f)$ (which can be empty because $a+d-e-f$ is not necessarily equal to $-2$). As long as we assign a unique letter representing the order of a zero or a pole to each marked point, the notation works for boundary strata in any $\oM_{0,k+n}$ and is more easily comparable with the corresponding enhanced level graph notation we will use for strata of $\oB_g(A;B)$, but the reader should be mindful that curves in these boundary strata of $\oM_{0,4}$ or $\oM_{0,5+n}$ do not carry, in general, any meromorphic differential with the indicated multiplicities.

\medskip

When pulled back to $H^2(\oM_{0,5+n})$ the above relation becomes
\begin{equation}\label{eq:WDVV(0,5)}
\tikz[baseline=-1mm]{\draw (A)--(B);\legm{A}{120}{a};\legm{A}{-180}{d};\legm{A}{-120}{-c};\legm{B}{-60}{-f};\legm{B}{60}{-e};\gg{0}{A};\gg{0}{B};}+ \tikz[baseline=-1mm]{\draw (A)--(B);\legm{A}{120}{a};\legm{A}{-120}{-c};\legm{B}{-60}{-f};\legm{B}{0}{d};\legm{B}{60}{-e};\gg{0}{A};\gg{0}{B};}=\tikz[baseline=-1mm]{\draw (A)--(B);\legm{A}{120}{a};\legm{A}{-180}{d};\legm{A}{-120}{-f};\legm{B}{-60}{-c};\legm{B}{60}{-e};\gg{0}{A};\gg{0}{B};}+\tikz[baseline=-1mm]{\draw (A)--(B);\legm{A}{120}{a};\legm{A}{-120}{-f};\legm{B}{-60}{-c};\legm{B}{0}{d};\legm{B}{60}{-e};\gg{0}{A};\gg{0}{B};} \in H^2(\oM_{0,5+n}),
\end{equation}
where we have omitted the $n$ marked legs with multiplicities $-C$, i.e. each term in the above formula represents a sum over all possible stable ways of distributing these $n$ legs between the two vertices of the graph.

\medskip

The preimage through the map \eqref{eq:projection} of each of the boundary strata intervening in equation~\eqref{eq:WDVV(0,5)} is a normal crossing divisor in $\oB_0(a,d,-e,-f;-c,-C)$ and, hence, is zero dimensional. By what explained in Section \ref{section:multiscale}, each such divisor is a union of strata $D_\Gamma$ with $\Gamma$ being either a one level connected graph with two vertices and one horizontal edge, or a two level connected graph with no horizontal edges. Strata of the second type can only be zero-dimensional if either there is one vertex at level $0$ and one at level $-1$, connected by a vertical edge, or there is one vertex at level $0$ connected by two vertical edges to exactly two vertices at level $-1$, with the $\mbC^*$-symmetry rescaling the differential on one component at level $-1$ with respect to the other being fixed by the global residue condition (2)(vi). This last situation can only happen if the component at level $0$ contains neither of the points marked with~$-e$ and~$-f$. Strata $D_\Gamma$ with $\Gamma$ being a two level graph with more than one connected component at level~$0$ or more than two connected components at level $-1$ are always empty: if they were not, chosen a differential in the stratum, we could produce at least a one-dimensional space of them by rescaling by a nonzero complex number the differential on one connected component with respect to the others on the same level (still satisfying the global residue condition (2)(vi) at level $-1$), which contradicts the fact that these boundary strata should be zero-dimensional. Moreover any of the three type of strata can be zero dimensional only when the differential on each component  of the stable curve exhibits no moduli, i.e. when it has two zeros and one pole with unconstrained residue or one zero and two poles with unconstrained residue (and, of course, any number of residueless poles).

\medskip

We will describe boundary strata of $\oB_0(a,d,-e,-f;-c,-C)$ using enhanced level graphs. In particular half-edges (not including legs) pointing upwards with respect to their vertex represent poles of order at least $2$, half-edges pointing downwards represent either zeros or points with multiplicity $0$ and horizontal half-edges represent simple poles for the meromorphic differential at the vertex. Somewhat similarly, legs pointing upwards represent poles (including simple poles) and legs pointing downwards represent zeros or points with multiplicity $0$. Zigzagged half-edges and legs always point up and represent residueless poles.

\medskip

Let us list all potentially non-empty components $D_\Gamma \subset  \oB_0(a,d,-e,-f;-c,-C)$ in the preimage of the stratum appearing in each term of equation \eqref{eq:WDVV(0,5)} by the corresponding enhanced level graph~$\Gamma$:
\begin{equation}\label{eq:terms 12}
\text{first term:}\quad\tikz[baseline=0]{\draw (A0)--(B1);\legm{A0}{120}{-e};\legm{A0}{60}{-f};\legm{B1}{-60}{d};\legm{B1}{-120}{a};\legmrl{B1}{150}{-c};\gg{0}{A0};\gg{0}{B1};\lab{A0}{-50}{5.6mm}{b-2};\lab{B1}{60}{4.8mm}{-b}},
\qquad\qquad\text{second term:}\quad
\tikz[baseline=0]{\draw (A0)--(0,0);\draw[decorate, decoration={zigzag, segment length=1.3mm,amplitude=0.5mm}] (0,0)--(B1);\legm{B1}{130}{-e};\legm{B1}{160}{-f};\legm{B1}{-90}{d};\legm{A0}{-150}{a};\legmrl{A0}{90}{-c};\gg{0}{A0};\gg{0}{B1};\lab{A0}{-50}{5.6mm}{b-2};\lab{B1}{60}{4.8mm}{-b}},
\end{equation}
where $b\ge 2$; the preimage of the third term is empty; and the fourth term gives
\begin{equation}\label{eq:terms 345}
\tikz[baseline=0]{\draw (A0)--(B1);\legm{A0}{120}{-e};\legm{B1}{160}{-f};\legm{B1}{-90}{a};\legm{A0}{-150}{d};\legmrl{A0}{60}{-c};\gg{0}{A0};\gg{0}{B1};\lab{A0}{-50}{5.6mm}{b-2};\lab{B1}{60}{4.8mm}{-b}},\qquad\qquad 
\tikz[baseline=0]{\draw (A0)--(B1);\legmrl{B1}{130}{-c};\legm{B1}{160}{-e};\legm{B1}{-90}{d};\legm{A0}{-150}{a};\legm{A0}{90}{-f};\gg{0}{A0};\gg{0}{B1};\lab{A0}{-50}{5.6mm}{b-2};\lab{B1}{60}{4.8mm}{-b}},\qquad\qquad \tikz[baseline=0]{\draw (A)--(C);\legm{A}{-90}{a};\legm{A}{90}{-f};\legmrl{C}{90}{-c};\legm{C}{-90}{d};\legm{C}{50}{-e};\gg{0}{A};\gg{0}{C};\lab{A}{20}{4.8mm}{-1};\lab{C}{160}{4.8mm}{-1}},
\end{equation}
where $b\ge 2$, together with
\begin{equation}
\tikz[baseline=0]{\draw (B1a)--(A0)--(C1b);\legmrl{A0}{90}{-c};\legm{C1b}{60}{-e};\legm{C1b}{-90}{d};\legm{B1a}{-90}{a};\legm{B1a}{120}{-f};\gg{0}{A0};\gg{0}{B1a};\gg{0}{C1b};\lab{C1b}{150}{4.9mm}{-b_2};\lab{B1a}{30}{4.9mm}{-b_1};\lab{A0}{-150}{6.9mm}{b_1-2};\lab{A0}{-30}{6.6mm}{b_2-2}},\qquad\qquad
\tikz[baseline=0]{\draw (B1a)--(A0)--(C1b);\legmrl{C1b}{60}{-c};\legm{C1b}{30}{-e};\legm{C1b}{-90}{d};\legm{B1a}{-90}{a};\legm{B1a}{120}{-f};\gg{0}{A0};\gg{0}{B1a};\gg{0}{C1b};\lab{C1b}{150}{4.9mm}{-b_2};\lab{B1a}{30}{4.9mm}{-b_1};\lab{A0}{-150}{6.9mm}{b_1-2};\lab{A0}{-30}{6.6mm}{b_2-2}},\label{eq:terms 67}
\end{equation}
where $b_1,b_2\ge 2$.

\medskip 

Notice in particular how the last two graphs above represent zero dimensional strata thanks to the fact, mentioned above, that the global residue condition (2)(vi) prescribes that the sum of the residues of the two differentials on the lower level at the poles with multiplicities $-b_1$ and $-b_2$ must vanish.

\medskip

To deduce from these considerations the intersection number of $\oH_0(a,d,-e,-f;-c,-C)$ with each boundary stratum appearing in equation  \eqref{eq:WDVV(0,5)} we need to study the multiplicity of each of these intersections. These can be computed by studying the local model for the undegeneration of a multiscale differential at a nodal singularity, see \cite[Section 4]{BCGGM18} and \cite[Section 4]{BCGGM19}.

\medskip

For all strata in~\eqref{eq:terms 12} and~\eqref{eq:terms 345} the local model and the argument are similar. Let us explain it for the first stratum in~\eqref{eq:terms 345}. Fix $I,J\subset \{1,\ldots,n\}$ with $I\sqcup J = \{1,\ldots,n\}$ and choose $p_1 \in \oH_0(d,b-2,-e;-c,-C_I)\subset \oM_{0,4+|I|}$ and $p_2 \in \oH_0(a,-f,-b;-C_J)\subset \oM_{0,3+|J|}$. Let $\sigma\colon\oM_{0,4+|I|} \times \oM_{0,3+|J|} \to \oM_{0,5+n}$ be the gluing map at the marked points carrying the labels~$b-2$ and $-b$. Let $p=\sigma(p_1,p_2)$ and denote $S:=\oH_0(a,d,-e,-f;-c,-C)$ for brevity.

\medskip

Choose local coordinates $V_1$ on $\oM_{0,4+|I|}$ and $V_2$ on $\oM_{0,3+|J|}$ so that $p_1=0\in V_1$  and $p_2=0\in V_2$. Denote by $\Delta_r \subset \mbC$ the disk of radius $r$. We claim that we can choose local coordinates $V_1\times V_2 \times \Delta_r$ on $\oM_{0,5+n}$ so that $S=\{0\}\times\{0\}\times \Delta_r$ and the image of $\sigma$ is $V_1\times V_2 \times \{0\}$. Then the transversality of the intersection is obvious. Let us describe how to choose these local coordinates.

\medskip

Let $\kappa := b-1$. The curves $C_1$ and $C_2$ corresponding, respectively, to the points $p_1$ and~$p_2$, carry meromorphic differentials $\alpha$ and $\beta$, both unique up to a multiplicative constant. For the meromorphic differential $\alpha$, we fix this constant arbitrarily. On the other hand, note that $\beta$ has a nonzero residue at the marked point labeled by $-b$. So for $\beta$, we fix the multiplicative constant by requiring that this residue is $1$. In a neighborhood of the marked points with labels $b-2$ and $-b$, respectively, there is a local coordinate $z$  on $C_1$ and $w$ on $C_2$ such that $\alpha = z^\kappa \frac{dz}{z}$ and $\beta = (-w^{-\kappa} + 1) \frac{dw}{w}$. Additionally to that, in a neighborhood of the marked point with label~$d$, there is a local coordinate $z_0$ on $C_1$ such that $\alpha = z_0^d dz_0$. We extend such local coordinates to curves in $V_1$ and $V_2$, possibly after shrinking $V_1$ and $V_2$, in an arbitrary way. So we can assume that the differentials $\alpha$ and $\beta$ are locally defined in these coordinates on all curves in $V_1$ and $V_2$.

\medskip

Now, given a curve $C_1$ in $V_1$, a curve $C_2$ in $V_2$, and a complex number $\epsilon\in \Delta_r$, there is a unique meromorphic differential on $C_1$ having exactly two poles at the marked points labeled by $-e$ and $b-2$, both of order $1$ and with residues $\epsilon^\kappa$ and $-\epsilon^\kappa$, respectively. Denote this differential by~$\eta$. If $r>0$ is small enough, then for any $\epsilon\in \Delta_r$ we can perturb the local coordinate $z$ (resp. $z_0$) on an annulus around the marked point labeled by $b-2$ (resp. $d$) in such a way that $\alpha+\eta=(z^\kappa-\epsilon^\kappa)\frac{dz}{z}$ (resp. $z_0^d dz_0$) on this annulus. Let us now remove the disk in $C_1$ bounded by the internal circle of the annulus around the marked point labeled by $b-2$, remove a neighborhood of the marked point $w=0$, and glue in the ``waist'' $zw=\epsilon$. Moreover, let us remove the disk bounded by the internal circle of the annulus around the marked point labeled by $d$ and glue it back in such a way that the differential $\alpha+\eta$ is equal to $z_0^d dz_0$ on the whole disk around the point $z_0=0$. When~$C_1$ and~$C_2$ correspond to $0\in V_1$ and $0\in V_2$, and $\epsilon\in \Delta_r \setminus \{0\}$, the curve thus obtained carries the differential that is glued from the differential~$\alpha+\eta$ on $C_1$ and the differential $\epsilon^\kappa \beta$ on $C_2$, which, around the ``waist'' $zw=\epsilon$, looks as follows:
$$
\epsilon^\kappa(-w^{-\kappa}+1)\frac{dw}{w}=(z^{\kappa}-\epsilon^\kappa) \frac{dz}{z}.
$$
Notice how the complex number~$\epsilon^\kappa$ corresponds to the period around the ``waist'' $zw=\epsilon$.

\medskip

We have thus shown that the intersection of $\oH_0(a,d,-e,-f;-c,-C)$ with the boundary strata appearing in equation \eqref{eq:WDVV(0,5)} at points in the image through $p$ of the zero-dimensional strata~\eqref{eq:terms 12} and~\eqref{eq:terms 345} is transversal, so its multiplicity is $1$.

\medskip

The computation of the multiplicity at points in the image through $p$ of the zero-dimensional strata \eqref{eq:terms 67} is similar. Let us focus on the first one. Fix $I,J,K \subset \{1,\ldots,n\}$ with $I\sqcup J\sqcup K = \{1,\ldots,n\}$ and choose $p_0 \in \oH_0(b_1-2,b_2-2;-c,-C_I)\subset \oM_{0,3+|I|}$, $p_1 \in \oH_0(a,-f,-b_1;-C_J)\subset \oM_{0,3+|J|}$, and $p_2 \in \oH_0(d,-e,-b_2;-C_K)\subset \oM_{0,3+|K|}$. Let $\sigma\colon\oM_{0,3+|I|}\times\oM_{0,3+|J|}\times\oM_{0,3+|K|} \to \oM_{0,5+n}$ be the gluing map at the marked points carrying the labels~$b_1-2$ and~$-b_1$ and the labels~$b_2-2$ and~$-b_2$, respectively. Let $p = \sigma(p_0,p_1,p_2)$ and denote $S:=\oH_0(a,d,-e,-f;-c,-C)$ for brevity.

\medskip

Choose local coordinates $V_0$ on $\oM_{0,3+|I|}$, $V_1$ on $\oM_{0,3+|J|}$, and $V_2$ on $\oM_{0,3+|K|}$ so that $p_0=0\in V_0$, $p_1=0\in V_1$, and $p_2=0\in V_2$. We will show that we can choose local coordinates $V_0\times V_1\times V_2 \times \Delta_r^2$ on $\oM_{0,5+n}$ so that $S$ is a curve in $\{0\}\times\{0\}\times\{0\}\times \Delta_r^2$ whose equation we will write down explicitly. Moreover, the image of $\sigma$ will be $V_0\times V_1 \times V_2 \times \{0\}\times \{0\}$, while $V_0\times V_1 \times V_2 \times \{0\} \times \Delta_r$ and $V_0\times V_1 \times V_2 \times \Delta_r \times \{0\}$ are the codimension $1$ boundary divisors containing undegenerations of only one of the two nodes, i.e. the images of the maps  $\sigma_1\colon\oM_{0,3+|J|}\times\oM_{0,4+|I|+|K|}\to \oM_{0,5+n}$ and $\sigma_2\colon\oM_{0,4+|I|+|J|}\times\oM_{0,3+|K|}\to \oM_{0,5+n}$, respectively.

\medskip

Let $\kappa_1 := b_1-1$ and $\kappa_2 := b_2-1$. The curves $C_0$, $C_1$, and $C_2$ corresponding, respectively, to the points $p_0$, $p_1$, and $p_2$, carry meromorphic differentials $\alpha$, $\beta_1$, and $\beta_2$, all three unique up to multiplicative constants. For the meromorphic differential $\alpha$, we fix this constant arbitrarily. On the other hand, note that $\beta_1$ (resp. $\beta_2$) has a nonzero residue at the marked point labeled by $-b_1$ (resp. $-b_2$). So for $\beta_1$ (resp. $\beta_2$), we fix the multiplicative constant by requiring that this residue is $1$ (resp. $-1$). In a neighborhood of the marked point with label $b_1-2$ there is a local coordinate $z_1$ on $C_0$ such that $\alpha = z_1^{\kappa_1} \frac{dz_1}{z_1}$, and in a neighborhood of the marked point with label~$b_2-2$ there is a local coordinate $z_2$ on $C_0$ such that $\alpha = z_2^{\kappa_2} \frac{dz_2}{z_2}$. In a neighborhood of the marked point with label $-b_1$, there is a local coordinate $w_1$ on $C_1$ such that $\beta_1 = (-w_1^{-\kappa_1} + 1) \frac{dw_1}{w_1}$. In a neighborhood of the marked point with label $-b_2$, there is a local coordinate $w_2$ on $C_2$ such that $\beta_2 = (-w_2^{-\kappa_2}-1)\frac{dw_2}{w_2}$. We extend such local coordinates to curves in $V_0$, $V_1$, and $V_2$ in an arbitrary way. So we can assume that the differentials $\alpha$, $\beta_1$, and $\beta_2$ are locally defined in these coordinates on all curves in $V_0$, $V_1$, and $V_2$.

\medskip

Now, given a curve $C_0$ in $V_0$, a curve $C_1$ in $V_1$, a curve $C_2$ in $V_2$, and a complex number $\epsilon$, there is a unique meromorphic differential $\eta_\epsilon$ on $C_0$ having exactly two poles at the marked points labeled by $b_1-2$ and $b_2-2$, both of order $1$ and with residues $-\epsilon$ and $\epsilon$, respectively. If $r>0$ is small enough, then for any $(\epsilon_1,\epsilon_2)\in\Delta_r^2$ we can perturb the local coordinate $z_1$ (resp.~$z_2$) on an annulus around the marked point labeled by $b_1-2$ (resp. $b_2-2$) in such a way that $\alpha+\eta_{\epsilon_1^{\kappa_1}}=(z_1^{\kappa_1}-\epsilon_1^{\kappa_1})\frac{dz_1}{z_1}$ (resp. $\alpha+\eta_{\epsilon_2^{\kappa_2}}=(z_2^{\kappa_2}+\epsilon_2^{\kappa_2})\frac{dz_2}{z_2}$) on this annulus. Let us now remove the disk in $C_0$ bounded by the internal circle of the annulus around the marked point labeled by $b_1-2$, remove a neighborhood of the marked point $w_1=0$, and glue in the ``waist'' $z_1w_1=\epsilon_1$. Let us also remove the disk in $C_0$ bounded by the internal circle of the annulus around the marked point labeled by $b_2-2$, remove a neighborhood of the marked point $w_2=0$, and glue in the ``waist'' $z_2w_2=\epsilon_2$. When $C_0$, $C_1$, and $C_2$ correspond to $0\in V_0$, $0\in V_1$, and $0\in V_2$, respectively, and also $\epsilon_1^{\kappa_1}=\epsilon_2^{\kappa_2}$, the curve thus obtained carries the differential that is glued from the differential $\alpha+\eta_{\epsilon_1^{\kappa_1}}$ on $C_0$ and the differentials $\epsilon_1^{\kappa_1}\beta_1$ and $\epsilon_2^{\kappa_2}\beta_2$ on $C_1$ and $C_2$, and which, around the ``waists'' $z_1w_1=\epsilon_1$ and $z_2w_2=\epsilon_2$, looks as follows:
\begin{gather*}
\epsilon_1^{\kappa_1}(-w_1^{-\kappa_1}+1)\frac{dw_1}{w_1}=(z_1^{\kappa_1}-\epsilon_1^{\kappa_1})\frac{dz_1}{z_1},\qquad \epsilon_2^{\kappa_2}(-w_2^{-\kappa_2}-1)\frac{dw_2}{w_2}=(z_2^{\kappa_2}+\epsilon_2^{\kappa_2})\frac{dz_2}{z_2}.
\end{gather*}
We see that $S$ is the curve in $\{0\}\times\{0\}\times\{0\}\times \Delta_r^2$ given by the equation $\epsilon_1^{\kappa_1}=\epsilon_2^{\kappa_2}$, which intersects the image of $\sigma_1$ with multiplicity $\kappa_2=b_2-1$ and the image of $\sigma_2$ with multiplicity $\kappa_1=b_1-1$.

\medskip

Notice that, for the case in question, we are interested in the intersection with the image of~$\sigma_1$ only, since the image of $\sigma_2$ is not a component of the stratum represented by the fourth term of equation \eqref{eq:WDVV(0,5)}. Indeed, contracting the edge labeled by $b_1-2$ and $-b_1$ in the first graph in \eqref{eq:terms 67}, we don't obtain the stable graph in the fourth term of equation $\eqref{eq:WDVV(0,5)}$. On the other hand, for the second graph in \eqref{eq:terms 67}, both the image of $\sigma_1$ and that of $\sigma_2$ are components of the stratum represented by the fourth term of equation \eqref{eq:WDVV(0,5)}, so the total multiplicity is the sum of the multiplicities of each component.

\medskip

We have thus shown that the intersection of $\oH_0(a,d,-e,-f;-c,-C)$ with the boundary strata appearing in equation \eqref{eq:WDVV(0,5)} at points in the image through $p$ of the zero-dimensional strata \eqref{eq:terms 67} has multiplcity $b_2-1$ and $b_1+b_2-2$, respectively.

\medskip

Let us slightly abuse the level graph notation above by having an enhanced level graph $\Gamma$ denote, instead, the pushforward to $H_*(\oM_{0,5+n})$, via the appropriate gluing map at the nodes, of $\prod_{v \in V(\Gamma)}[\oH(v)]$, where $V(\Gamma)$ is the set of vertices of $\Gamma$ and $\oH(v):= \oH_0(A(v);B(v))$, with~$A(v)$ and $B(v)$ being the labels at the straight and zigzagged half-edges or legs of $v$. The above considerations prove that the intersection of $[\oH_0(a,d,-e,-f;-c,-C)]$ with equation \eqref{eq:WDVV(0,5)} is equivalent to
\begin{align*}
\sum_{b\ge 2} \tikz[baseline=0]{\draw (A0)--(B1);\legm{A0}{120}{-e};\legm{A0}{60}{-f};\legm{B1}{-60}{d};\legm{B1}{-120}{a};\legmrl{B1}{150}{-c};\gg{0}{A0};\gg{0}{B1};\lab{A0}{-50}{5.6mm}{b-2};\lab{B1}{60}{4.8mm}{-b}} +&
\sum_{b\ge 2} \tikz[baseline=0]{\draw (A0)--(0,0);\draw[decorate, decoration={zigzag, segment length=1.3mm,amplitude=0.5mm}] (0,0)--(B1);\legm{B1}{130}{-e};\legm{B1}{160}{-f};\legm{B1}{-90}{d};\legm{A0}{-150}{a};\legmrl{A0}{90}{-c};\gg{0}{A0};\gg{0}{B1};\lab{A0}{-50}{5.6mm}{b-2};\lab{B1}{60}{4.8mm}{-b}}=
\sum_{b\ge 2} \tikz[baseline=0]{\draw (A0)--(B1);\legm{A0}{120}{-e};\legm{B1}{160}{-f};\legm{B1}{-90}{a};\legm{A0}{-150}{d};\legmrl{A0}{60}{-c};\gg{0}{A0};\gg{0}{B1};\lab{A0}{-50}{5.6mm}{b-2};\lab{B1}{60}{4.8mm}{-b}}+
\tikz[baseline=0]{\draw (A)--(C);\legm{A}{-90}{a};\legm{A}{90}{-f};\legmrl{C}{90}{-c};\legm{C}{-90}{d};\legm{C}{50}{-e};\gg{0}{A};\gg{0}{C};\lab{A}{20}{4.8mm}{-1};\lab{C}{160}{4.8mm}{-1}}+
\sum_{b\ge 2} \tikz[baseline=0]{\draw (A0)--(B1);\legmrl{B1}{130}{-c};\legm{B1}{160}{-e};\legm{B1}{-90}{d};\legm{A0}{-150}{a};\legm{A0}{90}{-f};\gg{0}{A0};\gg{0}{B1};\lab{A0}{-50}{5.6mm}{b-2};\lab{B1}{60}{4.8mm}{-b}}\\
&+\sum_{b_1,b_2\ge 2} (b_2-1) \tikz[baseline=0]{\draw (B1a)--(A0)--(C1b);\legmrl{A0}{90}{-c};\legm{C1b}{60}{-e};\legm{C1b}{-90}{d};\legm{B1a}{-90}{a};\legm{B1a}{120}{-f};\gg{0}{A0};\gg{0}{B1a};\gg{0}{C1b};\lab{C1b}{150}{4.9mm}{-b_2};\lab{B1a}{30}{4.9mm}{-b_1};\lab{A0}{-150}{6.9mm}{b_1-2};\lab{A0}{-30}{6.6mm}{b_2-2}}+
\sum_{b_,b_2\ge 2} (b_1+b_2-2) \tikz[baseline=0]{\draw (B1a)--(A0)--(C1b);\legmrl{C1b}{60}{-c};\legm{C1b}{30}{-e};\legm{C1b}{-90}{d};\legm{B1a}{-90}{a};\legm{B1a}{120}{-f};\gg{0}{A0};\gg{0}{B1a};\gg{0}{C1b};\lab{C1b}{150}{4.9mm}{-b_2};\lab{B1a}{30}{4.9mm}{-b_1};\lab{A0}{-150}{6.9mm}{b_1-2};\lab{A0}{-30}{6.6mm}{b_2-2}}.
\end{align*}
Notice that, all of the involved spaces of meromorphic differentials being zero-dimensional, the above equation is an equality of numbers. It is easy to see that, expressing this equality in terms of the two generating series $\theta^a_{b,c}(t_2,t_3,\ldots)$ and $P^{a,b}(t_2,t_3,\ldots)$, we obtain exactly the seven terms equation of Lemma~\ref{lemma:WDVV}.

\end{document}